\documentclass[11pt,onecolumn]{article}
\usepackage[top=1 in, bottom=1 in, left=0.75 in, right=0.75 in]{geometry}
\usepackage{hyperref}
\usepackage{amsmath} 
\usepackage{amssymb}
\usepackage{tabularx}
\usepackage[sc,osf,noBBpl]{mathpazo} 
\usepackage{cite} 
\usepackage{graphicx} 
\usepackage{subcaption}
\usepackage{algorithm}
\usepackage{algorithmic}
\usepackage{amssymb}
\usepackage{epsfig,xcolor,url}
\usepackage{epstopdf}
\usepackage{bbm}
\usepackage{enumitem} 

\usepackage{stfloats}
\usepackage{tikz}
\usetikzlibrary{arrows.meta,positioning}

\usepackage[utf8]{inputenc}
\usepackage{amsthm} 
\renewenvironment{proof}{{\sffamily\bfseries Proof. }}{\qed}
\newtheorem{theorem}{Theorem}
\newtheorem{lemma}{Lemma}

\theoremstyle{definition}

\newtheorem{assum}{Assumption}

\usepackage{thmtools}
\usepackage{thm-restate} 
\usepackage{enumitem}

\title{Steady-State Approximation Error of Heterogeneous Mean-Field Models}

\author{
	Lei Ying \\
	University of Michigan, Ann Arbor \\  \texttt{leiying@umich.edu}\\
}

\begin{document}
\date{}
\maketitle

\begin{abstract}
This paper studies heterogeneous mean-field models in which agent parameters are sampled from a population distribution. We establish an \(O(1/M)\) bound on the steady-state mean-square error between the occupancy measure of the $M$-agent system and the corresponding annealed mean-field equilibrium. The analysis extends Stein's method for homogeneous mean-field models and reveals a fundamental difference between homogeneous and heterogeneous systems. While stability of the mean-field dynamics is sufficient in the homogeneous setting, heterogeneous systems further require uniform robustness of the occupancy dynamics with respect to perturbations of the initial condition. The results are illustrated through a heterogeneous SIS epidemic model.

\end{abstract}

\section*{\normalsize The Role of Generative AI (GenAI) }
This paper is a collaborative effort between the author and GenAI. ChatGPT (GPT-5.5) was extensively used in the preparation of this manuscript. Most proofs were generated, corrected, and refined through iterative interactions between ChatGPT and the author. The author verified all results, organized and edited the paper, and takes full responsibility for its correctness and content.

\paragraph{Level of GenAI involvement:} Collaborative.

\section{Introduction}

Mean-field models provide tractable approximations for large-scale stochastic systems by replacing stochastic interactions among many agents with deterministic dynamics. They have been widely used in epidemic models, queueing and load-balancing systems, communication networks, and interacting particle systems. For homogeneous systems, \cite{Yin_16} established an \(O(1/M)\) bound on the steady-state mean-square approximation error for a broad class of mean-field models. The analysis is based on Stein's method \cite{BarChe_05} and shows that exponential stability of the mean-field dynamics plays a central role in determining the approximation error.

Many systems in practice are inherently heterogeneous, such as  an epidemic model with heterogeneous infection rates. For heterogeneous mean-field models, {finite-time} approximation results have  been established in \cite{AllGas_22}, which characterize the approximation error of the {marginal} state distribution of individual agents over a finite time horizon.  Our paper differs from~\cite{AllGas_22} in two
fundamental aspects.
\begin{itemize}

\item
\textbf{Finite-time versus steady-state approximation.}
The results in~\cite{AllGas_22} concern finite-time
approximation errors for process-level dynamics, pioneered in the seminal work of Kurtz~\cite{Kur_71,Kur_81}. This
paper studies steady-state approximation errors. In general,
finite-time convergence does not imply convergence of
stationary distributions, and establishing steady-state
convergence requires additional arguments related
to stability and interchange of limits. Therefore, a finite-time approximation bound does not directly apply to steady-state approximation. 

\item
\textbf{Marginal probabilities versus occupancy measures.}
The approximation results in~\cite{AllGas_22} are for marginal state probabilities of individual agents.
In contrast, this paper studies the convergence of the
population-level occupancy measure, which is central to
mean-field approximation.
To illustrate the distinction, consider \(M\) Bernoulli
random variables satisfying
\(W_1=W_m,\ \forall m,\)
so that they are fully correlated. In this case, the marginal probability of an agent is
\(\Pr(W_m=1)=p\)
 but the occupancy measure
\(X
=
\frac1M
\sum_{m=1}^M
\mathbf 1_{\{W_m=1\}}\)
is a Bernoulli random variable with parameter \(p\).
Therefore, accurate approximation of marginal
distributions of individual agents alone does not directly imply the same for 
the occupancy measure.

\end{itemize}

\subsection{Main Result}
Assume each agent evolves on a common finite state space
\(
\mathcal S=\{1,\dots,n\}.
\)
Let
\(
W_m(t)\in\{e_1,\dots,e_n\}\subset\mathbb R^n
\)
represent the state of agent \(m\) at time \(t\), where
\(e_i\) is the \(i\)-th standard basis vector.
The empirical occupancy measure is
\[
X_{\nu_M}(t)
=
\frac1M
\sum_{m=1}^M
W_m(t),
\] where $\nu_M$ is the empirical distribution of the agent types. 
This paper establishes sufficient conditions under which
\[
\mathbb E
\left[
\|X_{\nu_M}(\infty)-x^*\|^2
\right]
=
O\left(\frac1M\right),
\]
where \(X_{\nu_M}(\infty)\) denotes the stationary occupancy
measure of the heterogeneous $M$-system with empirical parameter distribution $\nu_M$, the expectation
is taken over both the stationary distribution and the
agent parameters, and \(x^*\) is the annealed
mean-field equilibrium defined in the next section. This approximation error implies that, for typical realizations
of the empirical type distribution \(\nu_M\), the occupancy
measure concentrates around the annealed equilibrium
\(x^*\) as \(M\to\infty\),
hence establishing the accuracy of the mean-field
approximation for large systems. 

We remark that in \cite{Yin_16}, the key condition underlying the $O(1/M)$ steady-state approximation error for homogeneous mean-field models is the exponential stability of the mean-field dynamics. For heterogeneous systems, stability alone is no longer sufficient. In addition, we require uniform robustness of occupancy measure dynamics with respect to perturbations of the initial condition. Intuitively, heterogeneous systems require not only convergence of trajectories toward equilibrium, but also uniform control of how trajectories vary across different agent types. Furthermore, to compare the finite heterogeneous system with the annealed mean-field equilibrium, we require the equilibrium map to be Lipschitz continuous with respect to the type distribution. This additional regularity condition is not needed in the homogeneous setting.  

\section{Models}

\subsection{Heterogeneous $M$-agent CTMC}

Consider a Continuous Time Markov Chain (CTMC) with $M$ heterogeneous agents. Each agent
$m$ is associated with a parameter $\theta_m\in \Theta$, independently
sampled from a continuous distribution $\nu$. Conditioning on
the realized parameters
\(\theta_1,\dots,\theta_M,\)
we define the empirical type distribution
\begin{equation}
\nu_M
=
\frac{1}{M}
\sum_{m=1}^M
\delta_{\theta_m},
\label{eq:nu}    
\end{equation}
where $\delta_{\theta_m}$ denotes the Dirac probability
measure concentrated at $\theta_m$. Since $\nu$ is continuous,
the parameters $\theta_m$ are distinct almost surely.

Each agent evolves on a finite state space
\(
\mathcal S=\{1,\ldots,n\}.
\)
 Let
\(
W(t)=(W_1(t),\ldots,W_M(t))
\) denote the state of the $M$-agent system, where
\(
W_m(t)\in\mathcal \{e_1, \cdots, e_n\}\subset \mathbb R^n,
\) is the one-hot vector representing the state of agent $m$ and $e_i$ is the $i$-th standard basis vector.
The empirical occupancy vector is
\[
X_{\nu_M}(t)
=
\frac1M
\sum_{m=1}^M
W_m(t),
\]
and the interaction field is
\[
Z_{\nu_M}(t)=\psi(X_{\nu_M}(t)).
\]

We assume that $W(t)$ is an irreducible CTMC with generator $G_M.$ Given $W(t)=W,$ let \(\mathcal R_M(W)\) denote the set of possible transitions from state
\(W\). For each transition \(r\in\mathcal R_M(W)\), let
\(q_r(W)\)
denote its transition rate and let
\(\Delta_r(W)\)
denote the corresponding state change. Define 
$$f_{\nu_M}(W)=\sum_{r\in\mathcal R_M(W)}
q_r(W) \Delta_r(W).
$$
We assume that $f_{\nu_M}(W)$ has the following structure
$$f_{\nu_M}(W)=(f\bigl(W_1,\psi(X(W));\theta_1\bigr), \cdots, f\bigl(W_M,\psi(X(W));\theta_M\bigr))^\top$$
where \[
X(W)
=
\frac1M
\sum_{m=1}^M
W_m
\] and
\[
f(w,z;\theta):
\Delta^{n-1}\times\mathbb R^d\times\Theta
\to
\mathbb R^n
\]
is a vector field such that for an agent with parameter $\theta$, when applying the generator component wise on $W_m,$
\[
(G_M W_m)(W)
=
f\bigl(W_m,\psi(X(W));\theta_m\bigr).
\]
The vector field $f$ represents the infinitesimal drift of agent $m$ induced by the finite CTMC and will be used to describe the mean-field dynamics later.  We assume for agent $m,$ it depends only on the state of agent $m,$ its parameter, and the empirical occupancy measure.

\subsection{Heterogeneous Mean-Field Models}
For each realized parameter \(\theta_m\), define
\(
w_{m}(t)
\in \Delta^{n-1}
\) as the mean-field approximation to the state-distribution vector of agent $m,$ i.e., $w_m(t)\approx \mathbb E[W_m(t)|\theta_m].$ We then define
$$x(t)=\frac1M\sum_{m=1}^Mw_m(t)\quad \hbox{and}\quad z(t)=\psi(x(t)).$$
We define two {\em deterministic} mean-field systems that approximate the evolution of $w_m(t)$: a quenched mean-field system and an annealed mean-field system.

\subsubsection{Quenched Mean-Field System}
Given realized parameters $\{\theta_1,\cdots, \theta_M\}$ and the realized
empirical distribution $\nu_M$ defined in \eqref{eq:nu},
the {\em quenched} mean-field system for agent $m$ is 
\begin{align*}
\dot{w}_m(t)
=f(w_m(t),z(t);\theta_m).
\end{align*} Let $w_{\nu_M}(t;\hat w)=(w_{\nu_M, 1}(t;\hat w),\dots,w_{\nu_M, M}(t;\hat w))^\top$ be the solution of the mean-field system above with initial condition $\hat w.$ 
The equilibrium of this quenched mean-field model is denoted by
\[
w_{\nu_M}^*
=
(w_{\nu_M,1}^*,\dots,w_{\nu_M,M}^*)^\top,
\]
where each equilibrium vector $w_{\nu_M,m}^*$ satisfies
\[
f(w_{\nu_M,m}^*,z_{\nu_M}^*;\theta_m)
=
0,\quad 
x_{\nu_M}^*
=
\frac1M
\sum_{m=1}^M
w_{\nu_M,m}^*, \quad\hbox{and}\quad
z_{\nu_M}^*
=
\psi\left(x_{\nu_M}^*\right).
\]

\subsubsection{Annealed Mean-Field System}

The quenched mean-field model depends on the realized
empirical distribution $\nu_M.$
The {\em annealed mean-field model} is a system that ``averages'' over the quenched systems.

For each parameter value \(\theta\in\Theta\), let
\(
w(t,\theta)
=
(w_1(t,\theta),\dots,w_n(t,\theta))^\top
\in\Delta^{n-1}
\)
denote the state-distribution vector associated with type
\(\theta\). The deterministic occupancy measure is defined by
\[
x(t)
=
\int_\Theta
w(t,\theta)\,\nu(d\theta),
\] and the corresponding interaction field is
\(
z(t)
=\psi(x(t)).
\)
The annealed mean-field model is 
\[
\dot w(t,\theta)
=
f(w(t,\theta), z(t);\theta),
\qquad \theta\in\Theta,
\]
and the annealed equilibrium is denoted by
$
w^*(\theta),$ which satisfies
\[
f(w^*(\theta), z^*;\theta)=0\ \forall \theta,\quad x^*
=
\int_\Theta
w^*(\theta)\,\nu(d\theta),\quad\hbox{and}\quad
z^*
=
\psi(x^*).
\]

\section{Main Results}
In this section, we establish the steady-state approximation error of both mean-field models. 
We first present some basic assumptions for the CTMCs to be considered in this paper. 
\begin{assum} We make the following assumptions throughout the paper. 
\begin{itemize}
\item
\textbf{A1 - Bounded transition-rate condition.} 
There exists a constant \(c>0\), independent of \(M\), such that the total jump intensity of the CTMC is bounded by \(cM\) uniformly over all system states and realized parameters.

\item
\textbf{A2 - Bounded jump condition.}
There exists a constant $\tilde c>0$, independent of $M$, such that each transition of the CTMC changes the states of at most $\tilde c$ agents. 

\item
\textbf{A3 - Smoothness condition.}
The vector field \(f(w,z;\theta)\) and the interaction map \(\psi(x)\) are twice continuously differentiable in their state arguments. Their first and second derivatives are uniformly bounded independently of \(M\), \(\nu_M\), and \(\theta\).
\end{itemize}
\end{assum}

\begin{theorem}[Steady-State Approximation Error]\label{thm:main}
Consider the heterogeneous \(M\)-agent system introduced in the previous section and the corresponding quenched and annealed mean-field models. Assume that A1-A3 and the following additional conditions hold.

\begin{itemize}
\item
\textbf{A4 - Occupancy stability.}
There exists a constant \(C_0>0\), independent of \(M\), \(\nu_M\), and \(\hat w\), such that
\[
\int_0^\infty
\|x_{\nu_M}(t;\hat w)-x_{\nu_M}^*\|^2\,dt
\le C_0,
\] where $x_{\nu_M}(t;\hat w)=\frac1M \sum_m w_{\nu_M,m}(t;\hat w)$. 

\item
\textbf{A5 - Occupancy robustness.}
Let \(\Delta\) denote the change in the full state vector  $w$ resulting from a single transition of the CTMC.
There exists a constant \(C_1>0\), independent of \(M\), \(\nu_M\), \(\hat w\), and \(\Delta\), such that
\[
\int_0^\infty
\left\|
\nabla_{\hat w}x_{\nu_M}(t;\hat w)[\Delta]
\right\|^2dt
\le
\frac{C_1}{M^2},
\]
and
\[
\int_0^\infty
\left\|
\nabla_{\hat w}^2x_{\nu_M}(t;\hat w)[\Delta,\Delta]
\right\|dt
\le
\frac{C_1}{M^2}.
\]
\end{itemize}

Then the steady-state approximation error with respect to the quenched mean-field equilibrium satisfies
\[
\mathbb E
\left[
\left\|
X_{\nu_M}(\infty)-x_{\nu_M}^*
\right\|^2
\,\middle|\,\nu_M
\right]
=
O\left(\frac1M\right).
\]

Furthermore, if the following condition also holds:
\begin{itemize}
\item
\textbf{A6 - Equilibrium regularity.}
There exists a distance measure \(\|\cdot\|_{\mathcal D}\) on probability measures over \(\Theta\) such that
\[
\mathbb E
\left[
\|\nu_M-\nu\|_{\mathcal D}^2
\right]
=
O\left(\frac1M\right),
\]
and the equilibrium map is Lipschitz continuous with respect to \(\|\cdot\|_{\mathcal D}\), i.e.,
\[
\|x_{\nu_M}^*-x^*\|
\le
L\|\nu_M-\nu\|_{\mathcal D},
\]
\end{itemize}
then the steady-state approximation error with respect to the annealed mean-field equilibrium satisfies
\[
\mathbb E
\left[
\left\|
X_{\nu_M}(\infty) - x^*
\right\|^2
\right]
=
O\left(\frac1M\right).
\]
\end{theorem}

\begin{proof}
Condition on the realized parameters \(\theta_1,\ldots,\theta_M\), so that the empirical distribution \(\nu_M\) is fixed. Recall the deterministic occupancy trajectory
\[
x_{\nu_M}(t;\hat w)
=
\frac1M\sum_{m=1}^M w_{\nu_M,m}(t;\hat w),
\]
where $w_{\nu_M}(0)=\hat w$ is the initial condition, and the Stein function
\begin{equation}
g_{\nu_M}(\hat w)
=
-\int_0^\infty
\left\|
x_{\nu_M}(t;\hat w)-x_{\nu_M}^*
\right\|^2dt. \label{stein-function}   
\end{equation}
By the occupancy stability condition A4, the integral in \eqref{stein-function} is finite, uniformly in \(M\), \(\nu_M\), and \(\hat w\) so the Stein function is well defined. 

Differentiating \(g_{\nu_M}\) along the quenched mean-field trajectory gives Stein's equation (or the Poisson equation):  
\[
\nabla_{\hat w} g_{\nu_M}(\hat w)\cdot f_{\nu_M}(\hat w)
=
\left\|
x(\hat w)-x_{\nu_M}^*
\right\|^2,
\] 
where 
\[f_{\nu_M}(\hat w)=(f(\hat{w}_1, \psi(x(\hat w));\theta_1),\cdots, f(\hat{w}_M, \psi(x(\hat w));\theta_M) )^\top\] and
\[
x(\hat w)=\frac1M\sum_{m=1}^M \hat w_m.
\] More details about the derivation above can be found in \cite{Yin_16}. 
Let \(G_M\) denote the generator of the finite \(M\)-agent CTMC. Since \(W(\infty)\) is distributed according to the stationary distribution of the finite CTMC, by definition of the steady-state, 
\[
\mathbb E\left[
G_M g_{\nu_M}(W(\infty))
\,\middle|\,\nu_M
\right]
=0.
\]
Therefore,
\begin{equation}
\mathbb E\left[
\left\|
X_{\nu_M}(\infty)-x_{\nu_M}^*
\right\|^2
\,\middle|\,\nu_M
\right]
=
\mathbb E\left[
\nabla g_{\nu_M}(W(\infty))\cdot f_{\nu_M}(W(\infty))
-
G_M g_{\nu_M}(W(\infty))
\,\middle|\,\nu_M
\right].    
\label{stein+ss}
\end{equation}

To prove the theorem, we now show that, uniformly in \(\hat w\),
\[
G_M g_{\nu_M}(\hat w)
=
\nabla g_{\nu_M}(\hat w)\cdot f_{\nu_M}(\hat w)
+
O\!\left(\frac1M\right).
\]
Let \(\Delta\) be the change in the full state vector caused by a single transition of the CTMC. By Taylor's theorem,
\[
g_{\nu_M}(\hat w+\Delta)
=
g_{\nu_M}(\hat w)
+
\nabla g_{\nu_M}(\hat w)\cdot \Delta
+
R(\hat w,\Delta),
\]
where
\begin{equation}
R(\hat w,\Delta)
=
\int_0^1
(1-s)
\nabla_{\hat w}^2 g_{\nu_M}(\hat w+s\Delta)[\Delta,\Delta]
\,ds    \label{eq:R}
\end{equation}
and 
\[\nabla_{\hat w}^2 g_{\nu_M}(\hat w+s\Delta)[\Delta,\Delta]
=
\left.
\frac{d^2}{d\varepsilon^2}
g_{\nu_M}(\hat w+s\Delta+\varepsilon\Delta)
\right|_{\varepsilon=0}.\]

We next compute the second derivative of \(g_{\nu_M}\) in the direction \(\Delta\).
Define
\[
\phi(\varepsilon)
=
g_{\nu_M}(\hat w+\varepsilon \Delta).
\]
Then
\[
\nabla_{\hat w}^2 g_{\nu_M}(\hat w)[\Delta,\Delta]
=
\phi''(0).
\]
Since
\[
g_{\nu_M}(\hat w)
=
-\int_0^\infty
\|x_{\nu_M}(t;\hat w)-x_{\nu_M}^*\|^2dt,
\]
we have
\[
\phi(\varepsilon)=g_{\nu_M}(\hat w+\varepsilon \Delta)
=
-\int_0^\infty
\left\|
x_{\nu_M}(t;\hat w+\varepsilon\Delta)
-
x_{\nu_M}^*
\right\|^2dt.
\]

Differentiating with respect to \(\varepsilon\),
\[
\phi'(\varepsilon)
=
-2
\int_0^\infty
\Bigl\langle
x_{\nu_M}(t;\hat w+\varepsilon\Delta)-x_{\nu_M}^*,
\nabla_{\hat w}
x_{\nu_M}(t;\hat w+\varepsilon\Delta)[\Delta]
\Bigr\rangle
dt,
\] where \[
\nabla_{\hat w}x_{\nu_M}(t;\hat w+\varepsilon\Delta)[\Delta]
=
\left.
\frac{d}{d\eta}
x_{\nu_M}\bigl(t;\hat w+\varepsilon\Delta+\eta\Delta\bigr)
\right|_{\eta=0}.
\] Note that we interchanged the differentiation  and integration based on Assumptions A4–A5 and the dominated convergence theorem.

Differentiating again and evaluating at \(\varepsilon=0\),
\begin{align*}
\phi''(0)
=\nabla_{\hat w}^2 g_{\nu_M}(\hat w)[\Delta,\Delta]=
-2
\int_0^\infty
\left\|
\nabla_{\hat w}
x_{\nu_M}(t;\hat w)[\Delta]
\right\|^2
dt
-2
\int_0^\infty
\Bigl\langle
x_{\nu_M}(t;\hat w)-x_{\nu_M}^*,
\nabla_{\hat w}^2
x_{\nu_M}(t;\hat w)[\Delta,\Delta]
\Bigr\rangle
dt.
\end{align*}
From Assumption A5, there exists $C_1$ independent of \(M\), \(\nu_M\), \(\hat w\), and \(\Delta\), such that
\[
\int_0^\infty
\left\|
\nabla_{\hat w}x_{\nu_M}(t;\hat w)[\Delta]
\right\|^2dt
\le
\frac{C_1}{M^2},
\]
and
\[
\int_0^\infty
\left\|
\nabla_{\hat w}^2x_{\nu_M}(t;\hat w)[\Delta,\Delta]
\right\|dt
\le
\frac{C_1}{M^2}.
\]
Since \(x_{\nu_M}(t;\hat w)\) and \(x_{\nu_M}^*\) are occupancy measures, by the triangle inequality, we have 
\[
\left\|
x_{\nu_M}(t;\hat w)-x_{\nu_M}^*
\right\|
\le 2.
\]
Therefore, there exists $C$ independent of \(M\), \(\nu_M\), \(\hat w\), and \(\Delta\), such that
\[
\left|
\nabla_{\hat w}^2 g_{\nu_M}(\hat w)[\Delta,\Delta]
\right|
\le
\frac{C}{M^2}.
\]
Consequently, according to the definition of $R(\hat w,\Delta)$ \eqref{eq:R}, 
\[
|R(\hat w,\Delta)|
\le
\frac{C}{M^2}.
\] 

Let \(\mathcal R_M(\hat w)\) denote the set of possible transitions from state
\(\hat w\). For each transition \(r\in\mathcal R_M(\hat w)\), let
\(q_r(\hat w)\)
denote its transition rate and let
\(\Delta_r(\hat w)\)
denote the corresponding state change. Then the generator of the finite CTMC is
\[
\begin{aligned}
(G_M g_{\nu_M})(\hat w)
=&
\sum_{r\in\mathcal R_M(\hat w)}
q_r(\hat w)
\Bigl(
g_{\nu_M}(\hat w+\Delta_r(\hat w))
-
g_{\nu_M}(\hat w)
\Bigr)\\
=& \sum_{r\in\mathcal R_M(\hat w)}
q_r(\hat w)
\left(\nabla g_{\nu_M}(\hat w)\cdot \Delta_r(\hat w)
+O\left(\frac{1}{M^2}\right)\right)\\
=_{(a)}&\nabla g_{\nu_M}(\hat w)\cdot f_{\nu_M}(\hat w) + O\left(\frac{1}{M^2}\right) \sum_{r\in\mathcal R_M(\hat w)}
q_r(\hat w)\\
=_{(b)}&\nabla g_{\nu_M}(\hat w)\cdot f_{\nu_M}(\hat w)
+
O\left(\frac1M\right).
\end{aligned}
\]
where equality $(a)$ holds due to the definition of $f_{\nu_M}$, and $(b)$ holds because $\sum_{r\in\mathcal R_M(\hat w)}
q_r(\hat w)=O(M)$ according to A1 and A2.  Substituting this into \eqref{stein+ss} leads to the $O(1/M)$ bound: 
\[
\mathbb E\!\left[
\left\|
X_{\nu_M}(\infty)-x_{\nu_M}^*
\right\|^2
\,\middle|\,\nu_M
\right]
=
O\!\left(\frac1M\right).
\]

Finally, under the equilibrium regularity condition A6,
\[
\mathbb E
\left[
\|x_{\nu_M}^*-x^*\|^2
\right]
=
O\!\left(\frac1M\right).
\]
Therefore,
\[
\mathbb E
\left[
\|X_{\nu_M}(\infty)-x^*\|^2
\right]
\le
2\mathbb E
\left[
\|X_{\nu_M}(\infty)-x_{\nu_M}^*\|^2
\right]
+
2\mathbb E
\left[
\|x_{\nu_M}^*-x^*\|^2
\right]
=
O\left(\frac1M\right).
\]
This completes the proof.
\end{proof}

\section{Application: A Heterogeneous SIS Model}

In this section, we apply our main result to a heterogeneous
SIS model population where each agent has an intrinsic
parameter $\theta_m$ which is independently drawn from the uniform distribution, i.e.,
\[
\theta_m
\overset{i.i.d.}{\sim}
\nu,
\qquad
\nu=\mathcal U[0,1].
\]

Each agent is in one of the two states, susceptible,
denoted by $S$, or infected, denoted by $I$. For the SIS model, we can use scalar $W_m(t)\in\{0,1\}$ to represent the state of agent $m$ such that $W_m(t)=1$ if agent $m$ is infected.  The empirical
occupancy measure can be determined by a scalar as well: 
\[
X_{\nu_M}(t)
=
\frac1M
\sum_{m=1}^M
W_m(t),
\] which is the fraction of infected agents. For each agent, the recovery time is exponential with mean one.
Each infected agent randomly selects another after waiting
for a random time that is exponential with
mean one. If the selected agent, say agent $m,$ is a susceptible node, it
becomes infected with probability $\theta_m$. Each susceptible node, after it becomes susceptible,
may also get infected by an external infection source after
a random time period that is exponentially  distributed with
mean $1/\alpha$.

\subsection{Mean-Field Models and Mean-Field Equilibrium}
Under the mean-field approximation, the infection rate experienced by agent $m$ is $\alpha+\theta_m z(t),$ where $z(t)=x_{\nu_M}(t)
=
\frac1M
\sum_{m=1}^M w_m(t).$
For a fixed realization of the empirical distribution
$\nu_M=
\frac1M
\sum_{m=1}^M
\delta_{\theta_m},
$
the {\bf quenched} mean-field model is
\[
\dot w_m(t)
=
(1-w_m(t))(\alpha+\theta_m z(t))
-
 w_m(t).
\]
The quenched mean-field equilibrium satisfies: 
\[x^*_{\nu_M}=
\frac1M
\sum_{m=1}^M
\frac{\alpha+\theta_m x^*_{\nu_M}}
{1+\alpha+\theta_m x^*_{\nu_M}}.
\]

The {\bf annealed} mean-field model is
\[
\dot w(t,\theta)
=
(1-w(t,\theta))(\alpha+\theta z(t))
-
w(t,\theta),
\qquad
\theta\in[0,1],
\]
where
\[
z(t)
=
\int_0^1 w(t,\theta)\,d\theta .
\]
The annealed mean-field equilibrium satisfies 
\[
x^*
=
\int_0^1
\frac{\alpha+\theta x^*}
{1+\alpha+\theta x^*}
\,d\theta .
\]

\begin{lemma}
The quenched equilibrium $x^*_{\nu_M}$ exists and is unique for every
realized empirical distribution $\nu_M$. The annealed mean field equilibrium exists and is unique as well.    
\end{lemma}

\begin{proof} We first consider the quenched mean-field system. Since $x_{\nu_M}(t)=z(t),$ we next prove that $z^*$ exists and is unique. 
Let
\[
z^*
=
\frac1M
\sum_{m=1}^M w_m^*.
\]
The equilibrium equation satisfies
\begin{equation}
0=(1-w_m^*)(\alpha+\theta_m z^*)-w_m^*.
\label{eq:sis-w}    
\end{equation}

Solving for $w_m^*$ gives
\[
w_m^*
=
\frac{\alpha+\theta_m z^*}
{1+\alpha+\theta_m z^*}.
\]
Therefore $z^*$ satisfies the fixed-point equation
\[
z^*=F_{\nu_M}(z^*),
\]
where
\[
F_{\nu_M}(z)
=
\frac1M
\sum_{m=1}^M
\frac{\alpha+\theta_m z}
{1+\alpha+\theta_m z}.
\]

For $z\in[0,1]$,
\[
0<F_{\nu_M}(z)<1.
\]
Furthermore,
\[
F_{\nu_M}'(z)
=
\frac1M
\sum_{m=1}^M
\frac{\theta_m}
{(1+\alpha+\theta_m z)^2}
\le
\frac1{(1+\alpha)^2}
<1.
\]
Hence $F_{\nu_M}$ is a contraction mapping on $[0,1]$.
By the Banach fixed-point theorem, there exists a unique
fixed point $z^*\in[0,1]$, which uniquely determines
\[
w_m^*
=
\frac{\alpha+\theta_m z^*}
{1+\alpha+\theta_m z^*}.
\]
Therefore the equilibrium exists and is unique. 

For the annealed model, the proof is essentially identical, with \[
F(z)
=
\int_0^1
\frac{\alpha+\theta z}
{1+\alpha+\theta z}
\,d\theta.
\]

\end{proof}

\subsection{Verifying the assumptions of Theorem \ref{thm:main}}
To establish the $O(1/M)$ steady-state approximation error, we now verify the required assumptions. It is straightforward to see that Assumptions A1-A3 hold. The following lemmas verify A4, A5 and A6. 

\begin{lemma}[Occupancy Stability (A4)]
Occupancy stability holds because the system is exponentially stable. 
\end{lemma}

\begin{proof}
Define
\[
\eta_m(t)
=
w_m(t)-w_{\nu_M,m}^*,
\]
and
\[
\bar\eta(t)
=
z(t)-z_{\nu_M}^*
=
\frac1M
\sum_{m=1}^M
\eta_m(t).
\]

Subtracting the equilibrium equation from the mean-field
equation gives
\[
\dot\eta_m(t)
=
-(1+\alpha+\theta_m z(t))\eta_m(t)
+
\theta_m(1-w_{\nu_M,m}^*)\bar\eta(t).
\]
Since
\(
1+\alpha+\theta_m z(t)
\ge
1+\alpha,
\)
and according to \eqref{eq:sis-w}, 
\[
\theta_m(1-w_{\nu_M,m}^*)
=
\frac{\theta_m}
{1+\alpha+\theta_m z_{\nu_M}^*}
\le
\frac1{1+\alpha},
\]
we obtain
\[
D^+|\eta_m(t)|
\le
-(1+\alpha)|\eta_m(t)|
+
\frac1{1+\alpha}
|\bar\eta(t)|,
\] where
\[
D^+f(t)
=
\limsup_{h\downarrow 0}
\frac{f(t+h)-f(t)}{h}
\]
denotes the upper right Dini derivative of \(f\). 

Define
\[
\eta_{\max}(t)
=
\max_{1\le m\le M}
|\eta_m(t)|.
\]
Then
\[
|\bar\eta(t)|
=
\left|
\frac1M
\sum_{m=1}^M
\eta_m(t)
\right|
\le
\eta_{\max}(t).
\]
Therefore
\[
D^+\eta_{\max}(t)
\le
-
\left(
(1+\alpha)-\frac1{1+\alpha}
\right)
\eta_{\max}(t).
\]

By the comparison lemma \cite{Kha_01},
\[
\eta_{\max}(t)
\le
e^{-\gamma t}
\eta_{\max}(0),
\]
where
\[
\gamma
=
(1+\alpha)-\frac1{1+\alpha}
=
\frac{\alpha(2+\alpha)}{1+\alpha}.
\]
Therefore, we obtain
\[
|x_{\nu_M}(t;\hat w)-x_{\nu_M}^*|^2
=
\left|
\frac1M\sum_{m=1}^M \eta_m(t)
\right|^2
\le
\eta^2_{\max}(t)
\le
e^{-2\gamma t}
\eta^2_{\max}(0).
\]

Since 
\[
|\eta_m(0)|=|w_m(0) -w_{\nu_M,m}^*| \le 1,
\]
we conclude
\[
\int_0^\infty
|x_{\nu_M}(t;\hat w)-x_{\nu_M}^*|^2dt
\le
\int_0^\infty e^{-2\gamma t}dt
=
\frac{1}{2\gamma}<\infty .
\]
Thus Assumption A4 holds.

\end{proof}

\begin{lemma}[Occupancy robustness (A5)]
The occupancy robustness
(A5) in Theorem~\ref{thm:main} holds.
\end{lemma}

\begin{proof}
Let \(\Delta\) be a single-transition direction of the CTMC. For the SIS model, \(\Delta\) changes the state of only one agent, so
\[
\frac1M\sum_{m=1}^M |\Delta_m| \le \frac{1}{M}.
\]

For notational convenience, define
\[
w_m^{(1)}(t)
=
\nabla_{\hat w}w_{\nu_M,m}(t;\hat w)[\Delta],
\qquad
x^{(1)}(t)
=
\nabla_{\hat w}x_{\nu_M}(t;\hat w)[\Delta]
=
\frac1M\sum_{m=1}^M w_m^{(1)}(t).
\]
Differentiating the SIS mean-field equation with respect to the perturbation parameter $\epsilon$ along direction $\Delta$ gives
\[
\dot w_m^{(1)}(t)
=
-(1+\alpha+\theta_m x(t))w_m^{(1)}(t)
+
\theta_m(1-w_m(t))x^{(1)}(t),
\]
Since \(0\le \theta_m\le 1\), \(0\le w_m(t)\le 1\), we have 
\[
D^+|w_m^{(1)}(t)|
\le
-(1+\alpha)
|w_m^{(1)}(t)|
+
\theta_m(1-w_m(t))
|x^{(1)}(t)|.
\]
Now let
\[
L_1(t)
=
\frac1M\sum_{m=1}^M |w_m^{(1)}(t)|.
\]
Since  \(|x^{(1)}(t)|\le L_1(t)\), we have
\[
D^+ L_1(t)
\le
-(1+\alpha)L_1(t)+L_1(t)
=
-\alpha L_1(t).
\]
Therefore,
\[
L_1(t)\le e^{-\alpha t}L_1(0).
\]
Since \(\Delta\) changes only one agent, 
\[
L_1(0)
=
\frac1M\sum_{m=1}^M |\Delta_m|
=
\frac{1}{M}
\]
and \begin{equation}
\label{eq:1st-bd}
|x^{(1)}(t)|
\le
L_1(t)
\le
\frac{ e^{-\alpha t}}{M},    
\end{equation}
which implies that
\[
\int_0^\infty |x^{(1)}(t)|^2dt
\le
\frac{1}{2\alpha M^2}.
\]

Next define
\[
w_m^{(2)}(t)
=
\nabla_{\hat w}^2w_{\nu_M,m}(t;\hat w)[\Delta,\Delta],
\qquad
x^{(2)}(t)
=
\nabla_{\hat w}^2x_{\nu_M}(t;\hat w)[\Delta,\Delta]
=
\frac1M\sum_{m=1}^M w_m^{(2)}(t).
\]
Differentiating the first-order system once more gives
\[
\dot w_m^{(2)}(t)
=
-(1+\alpha+\theta_m x(t))w_m^{(2)}(t)
+
\theta_m(1-w_m(t))x^{(2)}(t)
-
2\theta_m w_m^{(1)}(t)x^{(1)}(t).
\]
Also \(w_m^{(2)}(0)=0\) because the initial perturbation is linear in \(\Delta\). Since \(0\le \theta_m\le 1\), \(0\le w_m(t)\le 1\), we have 
\[
D^+ |w_m^{(2)}(t)|
\le
-(1+\alpha)|w_m^{(2)}(t)|
+
|x^{(2)}(t)|
+
2|w_m^{(1)}(t)||x^{(1)}(t)|.
\]
Let
\[
L_2(t)
=
\frac1M\sum_{m=1}^M |w_m^{(2)}(t)|.
\]
Using \(|x^{(2)}(t)|\le L_2(t)\), we obtain
\[
D^+L_2(t)
\le
-(1+\alpha)L_2(t)+L_2(t)
+
2|x^{(1)}(t)|L_1(t).
\]
Thus
\[
D^+L_2(t)
\le
-\alpha L_2(t)
+
2|x^{(1)}(t)|L_1(t).
\]
Using the bound \eqref{eq:1st-bd}, we have
\[
|x^{(1)}(t)|L_1(t)
\le
\frac{e^{-2\alpha t}}{M^2}.
\]
Therefore,
\[
D^+L_2(t)
\le
-\alpha L_2(t)
+
\frac{2e^{-2\alpha t}}{M^2}.
\]
Since \(L_2(0)=0\), the comparison lemma \cite{Kha_01} gives
\[
L_2(t)
\le\int_0^te^{-\alpha(t-s)} \frac{2e^{-2\alpha s}}{M^2} \ ds \le 
\frac{2e^{-\alpha t}}{\alpha M^2}.
\]
Hence
\[
\int_0^\infty |x^{(2)}(t)|dt
\le\int_0^\infty L_2(t)dt\le
\frac{2}{\alpha^2}\frac{1}{M^2}.
\]
Therefore, A5 holds by choosing $C_1=\frac{2}{\alpha^2}+\frac{1}{2\alpha}.$
\end{proof}

\begin{lemma}[Equilibrium regularity condition (A6)]
For the heterogeneous SIS model, define
\[
\|\rho-\eta\|_{\mathcal D}
:=
\sup_{z\in[0,1]}
\left|
\int_0^1
\frac{\alpha+\theta z}{1+\alpha+\theta z}
\,\rho(d\theta)
-
\int_0^1
\frac{\alpha+\theta z}{1+\alpha+\theta z}
\,\eta(d\theta)
\right|.
\]
Then the equilibrium map
is Lipschitz continuous with respect to
\(\|\cdot\|_{\mathcal D}\) and
\[
\mathbb E\|\nu_M-\nu\|_{\mathcal D}^2
=
O\left(\frac1M\right).
\]
\end{lemma}

\begin{proof}
For a type distribution \(\rho\), the equilibrium occupancy
\(x_\rho^*\) satisfies
\[
x_\rho^*
=
F_\rho(x_\rho^*),
\]
where
\[
F_\rho(z)
=
\int_0^1
\frac{\alpha+\theta z}{1+\alpha+\theta z}
\,\rho(d\theta).
\]
Define
\[
\phi(\theta,z)
=
\frac{\alpha+\theta z}{1+\alpha+\theta z}.
\]
For \(z\in[0,1]\),
\[
\left|
\frac{\partial \phi}{\partial z}(\theta,z)
\right|
=
\frac{\theta}{(1+\alpha+\theta z)^2}
\le
\frac1{(1+\alpha)^2}
=:a<1.
\]
Thus, for every probability measure \(\rho\),
\[
|F_\rho(z_1)-F_\rho(z_2)|
\le
a|z_1-z_2|.
\]

Let \(x_\rho^*\) and \(x_\eta^*\) be the equilibrium
occupancies associated with \(\rho\) and \(\eta\). Then by the definition of  $\cal D,$
\[
\begin{aligned}
|x_\rho^*-x_\eta^*|
&=
|F_\rho(x_\rho^*)-F_\eta(x_\eta^*)|
\\
&\le
|F_\rho(x_\rho^*)-F_\eta(x_\rho^*)|
+
|F_\eta(x_\rho^*)-F_\eta(x_\eta^*)|
\\
&\le
\|\rho-\eta\|_{\mathcal D}
+
a|x_\rho^*-x_\eta^*|.
\end{aligned}
\]
Therefore,
\[
|x_\rho^*-x_\eta^*|
\le
\frac1{1-a}
\|\rho-\eta\|_{\mathcal D}.
\]
Hence the equilibrium map is Lipschitz continuous.

It remains to verify the empirical-measure concentration
condition. Define
\[
\phi(\theta,z)
=
\frac{\alpha+\theta z}{1+\alpha+\theta z},
\quad\hbox{and}\quad
A_M(z)
=
\int_0^1
\phi(\theta,z)
(\nu_M-\nu) (d\theta).
\]
Then by definition of $\cal D,$
\[
\|\nu_M-\nu\|_{\mathcal D}
=
\sup_{z\in[0,1]}|A_M(z)|.
\]

Since
\(\phi(\theta,0)=\frac{\alpha}{1+\alpha}\)
is constant, we have
\(A_M(0)=0.\)
Therefore,
\[
A_M(z)
=
\int_0^z A_M'(u)\,du,
\] where
\[
A_M'(u)
=
\int_0^1
\frac{\theta}
{(1+\alpha+\theta u)^2}
(\nu_M-\nu)(d\theta),
\]
and hence
\[
\sup_{z\in[0,1]}|A_M(z)|
\le
\int_0^1 |A_M'(u)|\,du.
\]
By the Cauchy-Schwarz inequality,
\[
\|\nu_M-\nu\|_{\mathcal D}^2
\le
\left(
\int_0^1 |A_M'(u)|\,du
\right)^2
\le
\int_0^1 |A_M'(u)|^2\,du.
\]

For each fixed \(u\in[0,1]\), define
\(\kappa_u(\theta)
=
\frac{\theta}
{(1+\alpha+\theta u)^2},\) then
\[
A_M'(u)
=
\frac1M
\sum_{m=1}^M
\kappa_u(\theta_m)
-
\mathbb E[\kappa_u(\theta)],
\] where the expectation is taken over $\theta\sim {\cal U}[0,1].$
Since the \(\theta_m\)'s are i.i.d.,
\[
\mathbb E|A_M'(u)|^2
=
\frac1M
\operatorname{Var}(\kappa_u(\theta)).
\]
Moreover,
\(0\le
\kappa_u(\theta)
\le
\frac1{(1+\alpha)^2},\)
so
\[
\operatorname{Var}(\kappa_u(\theta))
\le
\mathbb E[\kappa_u(\theta)^2]
\le
\frac1{(1+\alpha)^4}.
\]
Thus,
\[
\mathbb E|A_M'(u)|^2
\le
\frac1{M(1+\alpha)^4}
\]
which implies that 
\[
\begin{aligned}
\mathbb E\|\nu_M-\nu\|_{\mathcal D}^2
&\le
\int_0^1
\mathbb E|A_M'(u)|^2\,du
\le
\frac1{M(1+\alpha)^4}.
\end{aligned}
\]
Hence, 
\(\mathbb E\|\nu_M-\nu\|_{\mathcal D}^2
=
O\left(\frac1M\right)\) 
and A6 holds.
\end{proof}

\section{Conclusion}

This paper established an \(O(1/M)\) bound on the mean-square steady-state approximation error of heterogeneous mean-field models. The analysis extends Stein's method for homogeneous systems and identifies occupancy stability and robustness as key conditions underlying the approximation error. The results were illustrated through a heterogeneous SIS epidemic model.

\section*{\large Acknowledgment}
The author would like to thank Prof. Weina Wang of Carnegie Mellon University for sharing her inspiring experience of using ChatGPT to tackle a long-standing open problem in queueing theory \cite{WanAmeChe_26}. Her experience motivated the author to explore whether ChatGPT could similarly be used to extend the $O(1/M)$ bound from homogeneous to heterogeneous mean-field systems, which led to this paper. 
\bibliographystyle{plain}
\bibliography{inlab-refs}
\end{document}